\documentclass {article}
\usepackage[english]{babel}
\usepackage[utf8]{inputenc}
\usepackage[T1]{fontenc}
\usepackage{graphicx,color}
\usepackage[final]{pdfpages}
\usepackage[extra]{tipa}
\usepackage{titlesec}
\usepackage{xspace}		
\usepackage[svgnames]{xcolor}
 \usepackage{lmodern}
\usepackage[hypertexnames=false]{hyperref}

\usepackage{mathtools}
\usepackage{graphicx}
\usepackage{amsmath,amsfonts}
\usepackage{amsthm}
\usepackage{dsfont}
\usepackage{amssymb}
\usepackage{mathenv}
\usepackage{nicefrac}

\usepackage{caption,subcaption}
\usepackage{color}
\usepackage{cite}
\mathtoolsset{showonlyrefs,showmanualtags}
\usepackage{aliascnt}

\newtheorem{theorem}{Theorem}
\newtheorem{proposition}{Proposition}
\newtheorem{corollary}{Corollary}
\newtheorem{lemma}{Lemma}
\newtheorem{definition}{Definition}

\newcounter{ex}
\theoremstyle{definition}
\newtheorem{example}[ex]{Example}

\newtheorem*{assumptions}{Assumptions (H)}

\newcommand{\D}{\partial}
\newcommand{\eps}{\varepsilon}
\newcommand{\R}{\mathbb{R}}
\newcommand{\la}{\lambda}

\begin{document}

\title{Maximum Principle and principal eigenvalue in unbounded domains under general boundary conditions}

\author{Samuel Nordmann\thanks{samnordmann@gmail.com ; School of Mathematical Sciences, Tel Aviv University.
}}



\maketitle
\begin{abstract}
This paper investigates the link between the Maximum Principle and the sign of the (generalized) principal eigenvalue for elliptic operators in unbounded domains. Our approach covers the cases of Dirichlet, Neumann, and (indefinite) Robin boundary conditions and treat them in a unified way. For a certain class of elliptic operators (including the class of selfadjoint operators), we establish that the positivity of the principal eigenvalue is a necessary and sufficient condition for the validity of the Maximum Principle. If the principal eigenvalue is zero, no general answer holds; instead, under a natural condition on the domain's size at infinity, we show that the operator satisfies what we call the Critical Maximum Principle. We also address the question of the simplicity of the principal eigenvalue, and a series of counterexamples is proposed to disprove some possible misconceptions. Our main results are new even for the more classical cases of Dirichlet boundary conditions and selfadjoint operators.
\end{abstract}
\paragraph{Keywords:} Elliptic equations, unbounded domains, principal eigenvalue, maximum principle, Dirichlet boundary conditions, Robin boundary conditions,\\

\noindent {\bf MSC 2010:} 35J15, 35B50, 35J25, 35P05.

\section{Introduction}

The Maximum Principle is the cornerstone property of elliptic operators and is related to several fundamental questions such as the existence of positive solutions, stability analysis, Liouville property, spectral analysis, the symmetry of solutions, etc. 
When the domain is bounded, it is classical that the positivity of the \emph{principal eigenvalue} of an elliptic operator is a necessary and sufficient condition for the validity of the Maximum Principle. When the domain is unbounded, one can introduce some notions of \emph{generalized principal eigenvalue} and study whether their signs provide necessary or sufficient conditions for the validity of the Maximum Principle~\cite{Berestycki2015b,Rossi2020}. Important literature is devoted to this question in under Dirichlet boundary conditions (see~\cite{Berestycki2015b,Pinchover2007,Nussbaum1992,Marg2018,Donsker1976}). The study of the Neumann and Robin cases in unbounded domains have been initiated in a recent article of Rossi~\cite{Rossi2020} but appears to be more incomplete.
Yet, these boundary conditions arise in various contexts (for instance, in the literature on reaction-diffusion equations and population dynamics~\cite{Berestyckib,Berestycki2013b}) and therefore the question of the validity of the Maximum Principle in these cases has important implications.

%It is known that, in general, the sign 

In this paper, we consider a classical generalization of the principal eigenvalue in unbounded domains and show that, under certain conditions, its positivity is a necessary or a sufficient condition for the validity of the Maximum Principle. Our approach covers the cases of non-selfadjoint elliptic operator under Dirichlet, Neumann, and Robin boundary conditions and addresses them all together in a unified way.

\paragraph*{}

Let us introduce our framework and notations. Let $\Omega\subset\R^n$ be a possibly unbounded domain and $\mathcal{L}$ be a linear elliptic operator of the form
\begin{equation}\label{DefinitionEllipticOperator}
\mathcal{L}u(x):=-\nabla\cdot\left(A(x)\cdot\nabla u(x)\right)-B(x)\cdot\nabla u(x)-c(x)u(x),\quad\forall x\in\Omega,
\end{equation}
where, $c:\Omega\to\R$, $B:\Omega\to\in\R^n$, and $A: \Omega \to \R^{n\times n}$.
We associate the operator $\mathcal{L}$ with boundary conditions
\begin{equation}
\mathcal{B}u=0,\qquad\forall x\in\D\Omega,
\end{equation}
which can be either of the Dirichlet type
\begin{equation}
\mathcal{B}u=\mathcal{B}^Du(x):=u(x),\qquad\forall x\in\D\Omega,
\end{equation}
or of the \emph{indefinite} Robin type
\begin{equation}\label{RobinBoundaryConditions}
\mathcal{B}u=\mathcal{B}^\gamma u(x):=\nu(x)\cdot A(x)\nabla u(x)+\gamma(x)u(x),\qquad\forall x\in\D\Omega,
\end{equation}
with $\gamma(\cdot):\D\Omega\to \R$ and where $\nu(x)$ is the outer normal direction of the domain at $x\in\D\Omega$.
%The operator $\D_{\nu_A}$ is the co-normal outward derivative through the boundary associated with $A$.
The term \emph{indefinite} indicates that we make no sign assumption on $\gamma$.
Neumann boundary conditions are obtained by taking $\gamma\equiv 0$ in~\eqref{RobinBoundaryConditions}.
Our standing assumptions are the following:
\begin{assumptions}\label{hyp:StandingAssumptions}
\mbox{}
\begin{itemize}
\item $A\in C^{1,\alpha}_{loc}\cap L^\infty(\overline{\Omega})$, $B,c\in C_{loc}^{0,\alpha}\cap L^\infty(\overline{\Omega})$, $\gamma\in C_{loc}^{1,\alpha}\cap L^\infty(\D\Omega)$ for some $\alpha\in(0,1)$,
\item the domain $\Omega$ is $C_{loc}^n$, or $C_{loc}^{2,\alpha}$ if $n=2$,
\item the operator $\mathcal{L}$ is elliptic, i.e., $A$ is symmetric positive definite (not necessarily uniformly in $x\in\Omega$), i.e., $A\geq \underline{a}(x)I_n $ for some function $\underline{a}:\Omega\to(0,+\infty)$.
\end{itemize}
\end{assumptions}
\noindent The regularity assumptions are required to apply some technical results from~\cite{Rossi2020}.

%The results in this article hold regardless of whether the boundary conditions are of the Dirichlet, Neumann, or Robin type.

\paragraph*{}
We define the notions of subsolutions, supersolutions and the Maximum Principle as follows.
\begin{definition}[sub/supersolutions and Maximum Principle]
\label{Def:subsuper_MP}
\mbox{}
\begin{itemize}
\item We say that $u \in C^{2,\alpha}_{loc}(\overline\Omega)$ is a subsolution (resp. supersolution) of $(\mathcal{L},\mathcal{B})$ in $\Omega$ if and only if
\begin{equation}\label{DefinitionSubSupersolution}
\left\{\begin{aligned}
&\mathcal{L}u\leq 0\  \text{(resp. $\geq0$)}&&\text{in }\Omega,\\
&\mathcal{B}u\leq0\  \text{(resp. $\geq0$)}&&\text{on }\D\Omega.
\end{aligned}\right.
\end{equation}

\item We say that $(\mathcal{L},\mathcal{B})$ satisfies the {Maximum Principle} in $\Omega$ if every subsolution with finite supremum is nonpositive.
\end{itemize}
\end{definition}

%Our study is mainly inspired by two articles by Berestycki, Rossi~\cite{Berestycki2015b} and Rossi~\cite{Rossi2020} dealing respectively with Dirichlet and Robin boundary conditions.
%In addition to our new results, we shall also rephrase and discuss some important results contained in these articles. 

%We also mention that our results echo with many of the studies devoted to the \emph{criticality} of elliptic operators, see~\cite{Simon1976,Pinchover2007,Marg2018}

%After recalling the definition of the generalized principal eigenvalue of an elliptic operator, 
% in terms of the sign of the generalized principal eigenvalue of the operator.

%We give in this paper several necessary and sufficient conditions for the validity of the Maximum Principle. 

When the domain is bounded, Krein-Rutman theorem implies the existence of what is called the principal eigenvalue of $(\mathcal{L},\mathcal{B})$~\cite{Daners2009}, that we denote by $\lambda_1$. This eigenvalue is real and has the lowest real part among all eigenvalues. It is then classical that the sign of $\lambda_1$ is equivalent to the validity of the Maximum Principle. Namely, if the domain $\Omega$ is bounded, then %, we have the following necessary and sufficient condition
\begin{equation}\label{EquivalenceClassiqueDomaineBornes}
\text{$(\mathcal{L},\mathcal{B})$ satisfies the Maximum Principle in $\Omega$ $\qquad \Leftrightarrow\qquad$ $\lambda_1>0$}.
\end{equation}
This result is essentially classical, at least for Dirichlet or positive Robin boundary conditions (see e.g.~\cite{Protter1984,GilbarDavid2015}). When there is no sign hypothesis on $\gamma$, the analysis often gets more difficult, mainly because in this case, the Laplace operator is not positive. However, it is noted by Daners~\cite{Daners2009} that any (indefinite) Robin problem can be re-written as a \emph{positive} Robin problem while preserving the structure of the operator. The results of Daners therefore entail the study of indefinite Robin boundary conditions using classical methods. In particular, they imply that the equivalence~\eqref{EquivalenceClassiqueDomaineBornes} holds in the case of indefinite Robin boundary conditions.

%Let us also mention that $\lambda_1$ is a real, simple eigenvalue whose real part is minimal among all eigenvalues. 
%In addition, it is the only eigenvalue associated with a positive eigenfunction.

%It is well know that if $c\leq0$ and $\gamma\geq0$, then $(\mathcal{L},\mathcal{B}^\gamma)$ satisfies the Maximum Principle in $\Omega$.
%However, this sufficient condition on the coefficients is quite restrictive and far from being necessary.
When the domain is unbounded, Krein-Rutman theorem is not applicable, however, one can still extend the definition of the principal eigenvalue and investigate the link between its sign and the validity of the Maximum Principle. The most standard and general definition of the principal eigenvalue in the unbounded setting is as follows.
\begin{definition}[generalized principal eigenvalue]\label{def:genralized_principal_eigenvalue}
We call \emph{generalized principal eigenvalue} of $(\mathcal{L},\mathcal{B})$ the quantity
\begin{equation}\label{DefinitionLambda1Robin}
\lambda_1:=\sup\left\{\lambda\in\R: \text{$(\mathcal{L}-\lambda,\mathcal{B})$ admits a positive supersolution}
 \right\}.
\end{equation}
\end{definition}
This definition coincides with the classical notion of principal eigenvalue given by Krein-Rutman theorem when applicable. In addition, it is knwon that $\lambda_1$ is associated with a positive eigenfunction~\cite{Berestycki2015b,Rossi2020}.
If the operator is self-adjoint, i.e., if $B\equiv0$ in~\eqref{DefinitionEllipticOperator}, the generalized principal eigenvalue defined in~\eqref{DefinitionLambda1Robin} can also be expressed through the Rayleigh-Ritz variational formula, namely,
under Robin boundary conditions
%\begin{equation}\label{RayleighFormulaEigen_Robin}
%\la_1
%=\inf\limits_{\substack{\psi\in H^1(\Omega)\\                                                                                                                                                                                                                                                                                                                                                                                                          \Vert \psi\Vert_{{L}^2}=1}} \mathcal{F}^\gamma(\psi):=\inf\limits_{\substack{\psi\in H^1(\Omega)\\                                                                                                                                                                                                                                                                                                                                                                                                          \Vert \psi\Vert_{{L}^2}=1}} \int_\Omega \vert \nabla \psi\vert_A^2-c\psi^2+\int_{\D\Omega}\gamma\psi^2,
%\end{equation}
\begin{equation}\label{RayleighFormulaEigen_Robin}
\la_1
=\inf\limits_{\substack{\psi\in H^1(\Omega)\\                                                                                                                                                                                                                                                                                                                                                                                                          \Vert \psi\Vert_{{L}^2}=1}} \int_\Omega \vert \nabla \psi\vert_A^2-c\psi^2+\int_{\D\Omega}\gamma\psi^2,
\end{equation}
where
$
\vert \nabla\psi\vert_A^2:=\nabla\psi\cdot A\nabla\psi,
$
and under Dirichlet boundary conditions
\begin{equation}\label{RayleighFormulaEigen_Dirichlet}
\la_1
=\inf\limits_{\substack{\psi\in H^1_0(\Omega)\\                                                                                                                                                                                                                                                                                                                                                                                                          \Vert \psi\Vert_{{L}^2}=1}} \int_\Omega \vert \nabla \psi\vert_A^2-c\psi^2,
\end{equation}
where $H^1_0(\Omega)$ is the space of $H^1$ functions which vanishes at the boundary $\D\Omega$.
%The characterization through the Rayleigh-Ritz formula is classical but we propose a proof in the Appendix~\ref{} for completeness.

Definition~\eqref{DefinitionLambda1Robin} has been used in many papers to study the validity of the Maximum principle under Dirichlet boundary conditions when the domain is nonsmooth~\cite{Berestycki1994,Pinchover2007a,Pinsky1995} or unbounded~\cite{Berestycki2015b,Rossi2020,Marg2018}, and coincides with previous variational caracterizations~\cite{Agmon1983,Nussbaum1992,Donsker1975,Donsker1976,Holland1978,Pinsky1995}.. Recently, Rossi~\cite{Rossi2020} used definition~\eqref{DefinitionLambda1Robin} for Robin boundary conditions in unbounded domains and laid the groundwork by proving important results, including the technical question of the existence of a positive eigenfunction associated with $\lambda_1$, see \autoref{lem:preliminary} below. Yet, the case of Robin boundary conditions is less understood than the Dirichlet case. We also mention the articles~\cite{Pinchover2002,Pinchover2020} in which an equivalent definition of the principal eigenvalue is considered for very general boundary condition, and also~\cite{Patrizi2008} which studies the Maximum Principle for fully nonlinear elliptic operator with Neumann boundary conditions.

In general, if the domain is unbounded, it is known that~\eqref{EquivalenceClassiqueDomaineBornes} does not hold, i.e., the positivity of the generalized principal eigenvalue $\lambda_1$ is neither a necessary nor a sufficient condition for the validity of the Maximum Principle, see Example~\ref{Ex:nonBoundedDrift} below and~\cite{Berestycki2015b,Rossi2020}. Alternative notions of generalized principal eigenvalues have therefore been proposed to provide, through their signs, such necessary and sufficient conditions.

Nevertheless, the definition of $\lambda_1$ through~\eqref{DefinitionLambda1Robin} is usually considered as the most natural generalization of the principal eigenvalue, firstly, because $\lambda_1$ is associated with an admissible eigenfunction (i.e. the ``$\sup$'' in~\eqref{DefinitionLambda1Robin} is actually a $\max$), secondly, because this definition matches with the Rayleigh-Ritz formula~\eqref{RayleighFormulaEigen_Robin}-\eqref{RayleighFormulaEigen_Dirichlet} in the case of a selfadjoint operator. It is therefore important to investigate the conditions under which the positivity of $\lambda_1$ ensures that the Maximum Principle holds.
%It is therefore important to understand under which conditions the sign  of $\lambda_1$ is a necessary 

\paragraph*{Outline.}
%In section~\ref{sec:Framework}, we present our framework, give some definitions, and state our main assumptions. 
In this paper, we study the links between the sign of $\lambda_1$ and the validity of the Maximum Principle.
Our results are stated and discussed in Section~\ref{sec:Results} and the proofs are given in Section~\ref{sec:Proofs}.

Our first result (\autoref{th:MPEigen}) establishes that, if the elliptic operator satisfies a certain condition~\eqref{AssumptionBoundedDrift} (which is automatically satisfied if the operator is selfadjoint), then the strict sign of the generalized principal eigenvalue is a sufficient and necessary condition for the validity of the Maximum Principle.

Then, we deal with the critical case where the generalized principal eigenvalue is zero. We show (\autoref{th:CriticalMP}) that under an additional condition on the growth of the domain~\eqref{GrowthConditionEigen}, the operator satisfies what we call the \emph{Critical Maximum Principle} (\autoref{def:Critical_MP}). 
From this, we derive a useful necessary and sufficient condition for the validity of the Maximum Principle in the critical case (\autoref{cor:VarphiBounded}) and deduce that no general answer holds.
We also address the question of the simplicity of the principal eigenvalue (\autoref{th:Simplicity}) and show that no general answer holds.
%, and a series of counterexample is given to disprove some possible misconceptions. 

Finally, \autoref{th:EigenMPNOTSelfAdjoint} provides a necessary and sufficient condition for the validity of the Maximum Principle for general elliptic operators that do not satisfy condition~\eqref{AssumptionBoundedDrift}. The sufficient condition involves the sign of an alternative notion of generalized principal eigenvalue.

\section{Statement of the results}\label{sec:Results}

Many of our statements deal with a certain class of elliptic operator whose drift term derives from a bounded potential, namely, we may assume that 
\begin{equation}\label{AssumptionBoundedDrift}
\exists\eta:{\Omega}\to\R\in C^1\cap L^\infty, \qquad \nabla\eta=-A^{-1}\cdot B.
\end{equation}
Before stating our results, let us point out some important cases where assumption~\eqref{AssumptionBoundedDrift} holds.
\begin{itemize}
\item Assumption~\eqref{AssumptionBoundedDrift} is automatically satisfied if the operator is self-adjoint, i.e., if $B=0$ in~\eqref{DefinitionEllipticOperator} (simply take $\eta\equiv 0$).
\item In dimension $n=1$, if $\frac{B}{A}$ has a bounded primitive, then~\eqref{AssumptionBoundedDrift} is satisfied.
\item If $A^{-1}\cdot B$ is constant, then \eqref{AssumptionBoundedDrift} reduces to
\begin{equation}
\sup\limits_{x\in\Omega}\left\vert A^{-1}\cdot B\cdot x \right\vert<+\infty.
\end{equation}
This assumption is not satisfied if $\Omega=\R^n$ and $B\neq0$, but it is satisfied if, for example, the domain is a cylinder $\Omega=\left\{(x_1,x')\in\R\times\R^{n-1}: \vert x'\vert< 1\right\}$ and $A^{-1}\cdot B$ is orthogonal to the $x_1$ direction.
\end{itemize}
\noindent
It was pointed to us by Professor Y. Pinchover (see also \cite[Remark~2 p.~103]{Pinsky1995}) that under assumption~\eqref{AssumptionBoundedDrift}, we can write
$\mathcal{L}u=-e^{-\eta}\nabla\cdot\left(e^{\eta}A\nabla u\right)+cu$, which implies that $\mathcal{L}$ is a selfadjoint operator in $L^2\left(\Omega,e^{\eta}dx\right)$. This fact is somehow used in our proofs, see \autoref{lem:1st} below.

\paragraph*{}
Our first result states that, under assumption~\eqref{AssumptionBoundedDrift}, the strict sign of $\lambda_1$ gives a necessary and sufficient condition for the validity of the Maximum Principle.
\begin{theorem}\label{th:MPEigen}
Assume that the standing assumptions $(H)$ hold.
\begin{enumerate}
\item Assume~\eqref{AssumptionBoundedDrift}. If $\lambda_1>0$ then $(\mathcal{L},\mathcal{B})$ satisfies the Maximum Principle in $\Omega$.
\item If $\lambda_1<0$ then $(\mathcal{L},\mathcal{B})$ does not satisfy the Maximum Principle in $\Omega$.
\end{enumerate}
\end{theorem}
%The first statement of~\autoref{th:MPEigen} under the sole assumption~\eqref{AssumptionBoundedDrift} is new ; however, in the special case where the operator is self-adjoint, the statement can be drawn from the results in~\cite{Berestycki2015b,Rossi2020}.
\noindent
The first statement of~\autoref{th:MPEigen} is new and seems not known even for the more classical case of a selfadjoint operator with Dirichlet boundary conditions.\footnote{In the case of a selfadjoint operator with Dirichlet (resp. Robin) boundary conditions, we deduce from~\cite[Theorem~1.7, $(i)$]{Berestycki2015b} (resp.~\cite[Theorem~2.9, $(iii)$]{Rossi2020}) that ``$\lambda_1>0$'' implies that there exists no bounded positive subsolution of $(\mathcal{L},\mathcal{B})$. The first statement of~\autoref{th:MPEigen} is therefore more general since it deals with possibly sign-changing subsolutions (with finite supremum).}
The second statement of \autoref{th:MPEigen} is already contained in~\cite[Theorem~2.9, $(ii)$]{Rossi2020} for the case of Robin boundary conditions, and in~\cite[Theorem~1.7, $(ii)$]{Berestycki2015b} for the case of Dirichlet boundary conditions.

If assumption~\eqref{AssumptionBoundedDrift} is not fulfilled, the first statement of \autoref{th:MPEigen} may not hold, as can be seen with the following example.
\begin{example}[Counterexample to \autoref{th:MPEigen} if~\eqref{AssumptionBoundedDrift} does not hold]
\label{Ex:nonBoundedDrift}
Set $\mathcal{L}u=-u''+\delta u'$ in $\Omega=\R$ for some parameter $\delta\in\R\setminus\{0\}$. By solving the equation $-\varphi''+\delta \varphi'=\lambda \varphi$ on $\R$ for any $\lambda\in\R$, we deduce that $\lambda_1=\frac{\delta^2}{4}>0$. However, the constant $u\equiv 1$ is a positive bounded subsolution, thus the Maximum Principle does not hold.
%\item For a parameter $\delta>0$, set $\mathcal{L}u=-u''+\frac{\delta}{x} u'$ in $\Omega=(1,+\infty)$ with Dirichlet boundary conditions $u(1)=0$. Then $u(x)=x^\alpha$ is a positive supersolution provided $\alpha\in(1,1+\delta)$. 
%\end{itemize}
\end{example}
\noindent Nervertheless, we give in \autoref{th:EigenMPNOTSelfAdjoint} below a positive result dealing with general elliptic operators which do not satisfy~\eqref{AssumptionBoundedDrift}.

We recall that the Maximum Principle deals with subsolutions with finite supremum (see \autoref{Def:subsuper_MP}). Without this condition, \autoref{th:MPEigen} does not hold:
\begin{example}[Counterexample to \autoref{th:MPEigen} when considering unbounded subsolutions]
Set $\mathcal{L}u=-u''+u$ on $\Omega=\R$. Classicaly, we have that $\lambda_1=1$. However, the function $x\mapsto e^{x}$ is a positive (unbounded) subsolution of the operator.
\end{example}

\autoref{th:MPEigen} does not deal with the case where $\lambda_1=0$. 
We will see in the sequel that, in this case, no general answer holds for the validity of the Maximum Principle.
However, let us mention already that one way to deal with the case $\lambda_1=0$ is the observation that the assumption ``$\lambda_1>0$'' in the first statement of~\autoref{th:MPEigen} can be replaced by a weaker assumption on the rate of convergence of the sequence of principal eigenvalues on truncated domains, see~\autoref{th:refinement} in Section~\ref{sec:ProofMPEigen}.

%Notons dès maintenant

%
%
%
%
%
%%While the Dirichlet case is now well understood, the case of Robin boundary conditions has been less studied,
%even though it appears naturally in many applications such as models for population dynamics (see e.g.~\cite{Bucur2017} for a survey on the Robin problem). 
%
%
%
%The reason is mostly technical. Contrarily to the case of Dirichlet boundary conditions, the spectrum of the Laplacian is not monotone with respect to domain inclusion~\cite{Giorgi2005a}, which makes difficult the approximation of the generalized principal eigenvalue value by ``classical'' principal eigenvalues defined on truncated domains. Recently, Rossi proved that such an approximation is possible using truncated domains endowed with mixed boundary conditions~\cite[Theorem~2.1]{Rossi2020}. 

\paragraph*{}

Let us now introduce the notion of principal eigenfunction.
\begin{definition}[principal eigenfunction]
We call \emph{principal eigenfunction} any function $\varphi$ which is positive on $\Omega$ and satisfies
\begin{equation}\label{EquationVarphi}
\left\{\begin{aligned}
&\mathcal{L}\varphi=\lambda_1 \varphi&\text{in }\Omega,\\
&\mathcal{B}\varphi=0&\text{on }\D\Omega.
\end{aligned}\right.
\end{equation}
\end{definition}
Principal eigenfunctions have been proved to exist by Berestycki and Rossi~\cite[Theorem~1.4]{Berestycki2015b} for Dirichlet boundary conditions and by
Rossi~\cite[Theorem~2.2]{Rossi2020} for Robin boundary conditions, see~\autoref{lem:preliminary} in Section~\ref{sec:Preliminary}.
However, principal eigenfunctions may not belong to $L^2(\Omega)$ or to $L^\infty$, which is why $\lambda_1$ is referred to as the \emph{generalized} eigenvalue. 
Also, in contrast with the case of bounded domains, $\lambda_1$ may not be simple (i.e., several linearly independent principal eigenfunctions may exist, see Example~\ref{ExampleNonSimple}).
%Let us also point out that in the case of indefinite Robin boundary conditions, Hopf's lemma in~\eqref{EquationVarphi} implies that $\varphi>0$ on $\overline{\Omega}$.
Let us also point out that, if (and only if) the domain is unbounded, then any $\lambda\in(-\infty,\lambda_1]$ is an eigenvalue which admits a positive eigenfunction, see \cite[Theorem~1.4]{Berestycki2015b} and~\cite[Theorem~2.2]{Rossi2020}.

We now define what we call the \emph{Critical Maximum Principle}.
\begin{definition}[Critical Maximum Principle]\label{def:Critical_MP}
We say that $(\mathcal{L},\mathcal{B})$ satisfies the {Critical Maximum Principle} in $\Omega$ if there exists a principal eigenfunction $\varphi$ such that every subsolution with finite supremum is either nonpositive or a constant multiple of $\varphi$.
\end{definition}
When the domain $\Omega$ is bounded, then $(\mathcal{L},\mathcal{B})$ satisfies the {Critical Maximum Principle} in $\Omega$ if and only if $\lambda_1\geq0$ (this is a direct consequence of the classical Strong Maximum Principle). The following result states that this property still holds under~\eqref{AssumptionBoundedDrift} if the domain is unbounded but satisfies a certain growth condition at infinity.
\begin{theorem}\label{th:CriticalMP} Assume that the standing assumptions~(H) hold.
\begin{enumerate}
\item Assume that~\eqref{AssumptionBoundedDrift} holds and that $\Omega$ satisfies
\begin{equation}\label{GrowthConditionEigen}
\left\vert \Omega\cap\{\vert x\vert\leq R\} \right\vert= O(R^2)\quad\text{ when }R\to+\infty.
\end{equation}
If $\lambda_1\geq0$ then $(\mathcal{L},\mathcal{B})$ satisfies the {Critical Maximum Principle} in $\Omega$.
\item If $\lambda_1<0$ then $(\mathcal{L},\mathcal{B})$ does not satisfy the {Critical Maximum Principle} in~$\Omega$.
\end{enumerate}
\end{theorem}

The first statement of \autoref{th:CriticalMP} gives a useful sufficient condition for the Critical Maximum Principle to hold. This result is new even for the case of Dirichlet boundary conditions and selfadjoint operators. It is nonetheless closely related to the property of \emph{criticality}~\cite{Simon1976} which are known to hold for Schrödinger operators in $\R^n$ if and only if $n\leq 2$, see~\cite{Pinchover2007,Pinsky1995,Pinchover2007a,Pinchover2020} and references therein. Note that this standard restriction on the dimension translates in \autoref{th:CriticalMP} to condition~\eqref{GrowthConditionEigen} on the size of the domain.

If condition~\eqref{AssumptionBoundedDrift} is not fulfilled, the first statement of \autoref{th:CriticalMP} does not hold in general, as can be seen through Example~\ref{Ex:nonBoundedDrift}. Let us give the following other example which shows that condition~\eqref{AssumptionBoundedDrift} is actually sharp in dimension $n=1$. We recall that, in dimension $n=1$, condition~\eqref{AssumptionBoundedDrift} reduces to assuming that $-\frac{B}{A}$ has a bounded primitive.
\begin{example}[Sharpness of~\eqref{AssumptionBoundedDrift} in \autoref{th:CriticalMP}]
Let us consider $\mathcal{L}u=-u''+\frac{2}{x}u'$ in $(1,+\infty)$ with the Neumann boundary condition $-u'(1)=0$. Since $\varphi\equiv1$ is a positive (super)solution,  we have that $\lambda_1\geq 0$. However, the function $v(x)= 1-\eps-\frac{1}{x}$ is a subsolution which changes sign provided $0<\eps\ll1$, hence the Critical Maximum Principle does not hold.
\end{example}

Condition~\eqref{GrowthConditionEigen} on the domain's size at infinity echoes with the assumptions of a celebrated Liouville theorem from~\cite{Berestycki1997b}. This condition turns out to be essentially sharp in our context, see the discussion at the end of Section~\ref{sec:Proof_CriticalMP} for more details. Let us simply show here that the first statement of \autoref{th:CriticalMP} does not hold for the Laplace operator in the entire space when the dimension is strictly greater than two.
\begin{example}[Counterexample to \autoref{th:CriticalMP} when~\eqref{GrowthConditionEigen} is not fulfilled]
Assume $n\geq3$, set $\Omega=\R^n$, $\mathcal{L}=-\Delta$, and let $\rho:\R^n\to \R$ be a smooth function which is nonnegative, non identically zero, and compactly supported. Let $\phi(x)=C\Vert x\Vert^{2-n}$ (with $C>0$) be the fundamental solution of the Laplacian in $\R^n$, and set $u:= -\phi\star \rho+K$ where $\star$ denotes the usual convolution product and where $K>0$ is a suitably large constant so that $u\geq0$. Then we have $\mathcal{L} u =-\rho\leq 0$, therefore $u$ is a positive bounded subsolution of the Laplace operator in $\R^n$ for which $\lambda_1=0$. This shows that one can construct positive non-colinear subsolutions, and so that \autoref{th:CriticalMP} is not satisfied.
\end{example}

We point out that if the Critical Maximum Principle holds but not the Maximum Principle, then we necessarily have $\lambda_1=0$. To see this, simply note that, in this case, the second statement of~\autoref{th:CriticalMP} implies $\lambda_1\geq0$, while the existence of a nontrivial subsolution which is also a multiple of a principal eigenfunction implies $\lambda_1\leq0$.

\paragraph*{}
As a consequence of~\autoref{th:CriticalMP}, we show that, in the critical case when $\lambda_1=0$, the existence of a \emph{bounded} principal eigenfunction is a necessary and sufficient condition for the validity of the Maximum Principle.
\begin{proposition}\label{cor:VarphiBounded}
Assume that~(H),~\eqref{AssumptionBoundedDrift},~\eqref{GrowthConditionEigen} hold, and that $\lambda_1=0$. Then $(\mathcal{L},\mathcal{B})$ satisfies the Maximum Principle if and only if there exists an unbounded principal eigenfunction.
\end{proposition}
This result implies that no general answer holds for the validity of the Maximum Principle when $\lambda_1=0$. This is illustrated on the following examples.
\begin{example}[Non-validity of the Maximum Principle when $\lambda_1=0$]\label{ex:1}
\mbox{}
\begin{itemize}
\item if the domain is bounded, it is classical that the Maximum Principle does not hold if $\lambda_1=0$. Indeed, the principal eigenfunction (given by Krein-Rutman theorem) is a positive bounded (sub)solution of $(\mathcal{L},\mathcal{B})$.
\item the principal eigenvalue of $\mathcal{L}u=-u''$ in $\R$ is $\lambda_1=0$ and the principal eigenfunctions are the constant functions, therefore \autoref{cor:VarphiBounded} implies that the Maximum Principle does not hold in this case.
\end{itemize}
\end{example}

\begin{example}[Validity of the Maximum Principle when $\lambda_1=0$]\label{ex:2}
Consider the operator $\mathcal{L}u=-u''$ in $\Omega=(0,+\infty)$ associated with the boundary condition $\mathcal{B}u=- u'(0)+\gamma u(0)$ for some positive constant $\gamma$. We see that $\lambda_1=0$ and that any eigenfunction is a positive multiple of $\varphi(x)=x+\frac{1}{\gamma}$. The principal eigenfunction $\varphi$ is unbounded, therefore \autoref{cor:VarphiBounded} implies that the Maximum Principle holds.
The case of Dirichlet boundary conditions is also covered by this example by taking $\varphi(x)=x$.
\end{example}

\paragraph*{}
Another consequence of~\autoref{th:CriticalMP} is the following sufficient condition for the simplicity of $\lambda_1$.
\begin{proposition}\label{th:Simplicity}
Assume~(H),~\eqref{AssumptionBoundedDrift},~\eqref{GrowthConditionEigen}, and that there exists a \emph{bounded} principal eigenfunction. Then  $\lambda_1$ is simple, i.e., the solution of~\eqref{EquationVarphi} is unique up to a multiplicative constant.
\end{proposition}
Note that, in~Example~\ref{ex:2}, we see that the principal eigenvalue is simple even though there exists no \emph{bounded} eigenfunction. It shows that the existence of a bounded eigenfunction is a sufficient condition but not a necessary condition for the simplicity of the generalized principal eigenvalue.

Let us now show that no general answer holds for the simplicity of $\lambda_1$.
Example~\ref{ex:1} exhibits situations where~\autoref{th:Simplicity} applies, and thus where $\lambda_1$ is simple. The following example, inspired from~\cite[Proposition~8.1]{Berestycki2015b}, is an instance where the principale eigenvalue is not simple.
\begin{example}[Non-simplicity of $\lambda_1$]
\label{ExampleNonSimple}
Assume $\Omega=\R$ and set $\mathcal{L}u:=-u''-c(x)u$ with $c<0$ in $(-1,1)$ and $c=0$ outside. 
Let us show that $\lambda_1=0$. On the one hand, the constant $1$ is a supersolution of $\mathcal{L}$ in $\R$, thus we have $\lambda_1\geq0$ from definition~\eqref{DefinitionLambda1Robin}. On the other hand, since the mean value of $c$ over $\R$ is $0$, we deduce $\lambda_1\leq 0$ from~\cite[Theorem~2.7]{Rossi2020}.
Let $u_-$ and $u_+$ be the solutions to $\mathcal{L} u=0$ in $\R$ satisfying $u_\pm(\pm1)=1$, $u_\pm'(\pm1)=0$. Using standard ODE arguments, we have that
$u_-$ is positive, nondecreasing, nonconstant, and identically equal to $1$ in $(-\infty,-1)$, whereas $u_+$ is positive, nonincreasing, nonconstant, and identically equal to $1$ in $(1,+\infty)$.
Therefore $u_\pm$ are two linearly independent principal eigenfunctions and so $\lambda_1=0$ is not simple.
\end{example}
Let us emphasize that if $\lambda_1$ is not simple, then one can construct a sign-changing solution of~\eqref{EquationVarphi} by a linear combination of two principal eigenfunctions. Thus, in contrast with the case of a bounded domain, $\lambda_1$ may admit sign-changing eigenfunctions.

%
%
% which has been proved to exist by Berestcyki and Rossi. 
%\begin{theorem}[\cite{Berestycki2015b,Rossi2020}]\label{PropositionEigenfunctionVariation}
%There exists $\varphi_1\in C^2(\overline{\Omega})$ which is positive and satisfies
%
%Any positive solution of this equation will be referred to as a \emph{principal eigenfunction}.
%\end{theorem}

%The reason is that the principal eigenfunction $\varphi$ given by \autoref{PropositionEigenfunctionVariation} may not be bounded or uniformly positive in $\Omega$. Instead, necessary and sufficient conditions are given by the sign of some other quantities which are defined as variations of~\eqref{DefinitionLambda1Robin}, namely, when adding extra conditions on the supersolution $\phi$ in~\eqref{DefinitionLambda1Robin}, for instance, to be bounded or to have a positive infimum. This holds for both Dirichlet and Robin boundary conditions. 

\paragraph*{}
All the above results deal with the class of elliptic operator that fulfill assumption~\eqref{AssumptionBoundedDrift}. 
Let us finally discuss the general case where \eqref{AssumptionBoundedDrift} does not hold. In this case, the positivity of $\lambda_1$ is not a sufficient condition for the validity of the Maximum Principle, as can be seen through Example~\ref{Ex:nonBoundedDrift}. However, under Dirichlet boundary conditions, Berestycki and Rossi~\cite[Theorem 1.6]{Berestycki2015b} show that a sufficient condition is given by the positivity of
\begin{equation}\label{DefLambdaTilde}
\tilde\lambda_1:=\sup\left\{\lambda\in\R: (\mathcal{L}-\lambda,\mathcal{B})\text{ admits a supersolution with positive infimum}\right\}.
\end{equation}
Note that the definition of $\tilde\lambda$ in~\eqref{DefLambdaTilde} differs from that of $\lambda$ in~\eqref{def:genralized_principal_eigenvalue} since we impose that the supersolution has positive infimum rather than assuming that it is positive. 
In general, we have $\tilde \lambda_1\leq \lambda_1$, but whether the equality holds is an open question, see~\cite[Conjecture~1]{Berestycki2015b}. Note that the existence of a principal eigenfunction with positive infimum implies $\tilde \lambda_1=\lambda_1$, but the converse is not true, see \cite[Proposition~8.1]{Berestycki2015b}.

The following theorem establishes that the positivity of $\tilde\lambda_1$ is a sufficient condition for the Maximum Principle to hold. In turn, the following result extends~\cite[Theorem 1.6]{Berestycki2015b} which deals with Drichlet boundary conditions to the case of indefinite Robin boundary conditions.
\begin{theorem}
\label{th:EigenMPNOTSelfAdjoint}
Assume $\Omega\subset\R^n$ is uniformly $C^{2}$, that $B,c,\gamma$ are uniformly $C^{0,\alpha}$, that $A$ is uniformly $C^{1,\alpha}$ and is uniformly elliptic (i.e. $A(x)\geq\underline{A}I_n$ for some constant $\underline{A}>0$). Recall $\tilde\lambda_1$ defined in~\eqref{DefLambdaTilde}.

If $\tilde \lambda_1>0$ then $(\mathcal{L},\mathcal{B})$ satisfies the Maximum principle in $\Omega$.
\end{theorem}
\noindent
The assumption that the domain is uniformly $C^2$ is defined as follows: there exist $R,C>0$ such that for all $x\in\D\Omega$, there exists some function $g:\R^{n-1}\to\R$ such that $\Vert g\Vert_{C^2}\leq C$ and
\begin{equation}\label{UniformlyC2}
\Omega\cap\{\vert x\vert\leq R\}=\left\{(x',x_n):x_n>g(x')\right\}\cap\{\vert x\vert\leq R\},
\end{equation}
in some system of coordinate. In particular, it implies a uniform interior ball condition.

The proof of \autoref{th:EigenMPNOTSelfAdjoint} can be adapted without difficulty to the case of oblique boundary conditions considered in~\cite{Rossi2020}.

%
%We show that, under the same conditions, $\lambda_1\geq0$ implies the validity of the Critical Maximum Principle.
%\begin{theorem}\label{th:EigenMPNOTSelfAdjoint}
%Assume $\Omega\subset\R^n$ is uniformly $C^{2}$, see~\eqref{UniformlyC2}, and that $A$ is uniformly elliptic, i.e., $A(x)\geq\underline{A}I_n$ for some constant $\underline{A}>0$. 
%
%If $\lambda_1\geq 0$ and if there exists a principal eigenfunction with positive infimum, then the Maximum principle holds 
%for $(\mathcal{L},\mathcal{B})$ in $\Omega$.
%\end{theorem}

%
%
%Rossi~\cite{Rossi2020} and Berestycki, Rossi~\cite{Berestycki2015b} make the important observation that, in contrast with the case of a bounded domain, the sign of $\lambda_1$ is neither a necessary nor a sufficient condition for the validity of the Maximum Principle in general. 
%
%
%
%The reason is that the principal eigenfunction $\varphi$ given by \autoref{PropositionEigenfunctionVariation} may not be bounded or uniformly positive in $\Omega$. Instead, necessary and sufficient conditions are given by the sign of some other quantities which are defined as variations of~\eqref{DefinitionLambda1Robin}, namely, when adding extra conditions on the supersolution $\phi$ in~\eqref{DefinitionLambda1Robin}, for instance, to be bounded or to have a positive infimum. This holds for both Dirichlet and Robin boudary conditions~. 
%
%In this paper, we show that in some important cases the sign of $\lambda_1$ does provide a sufficient (sometimes necessary) condition for the validity of the Maximum Principle.
%

\section{Proofs of the results}\label{sec:Proofs}

\subsection{Preliminary -- existence of a principal eigenfunction}\label{sec:Preliminary}

Denote by $B_R$ the ball of radius $R>0$ and define $\lambda^R_1$ the ``classical'' principal eigenvalue of the truncated eigenvalue problem
\begin{equation}\label{TruncatedProblem}
\left\{\begin{aligned}
&-\mathcal{L}\varphi^R=\lambda_1^R\varphi^R &&\text{in }\Omega\cap B_R,\\
& \mathcal{B}\varphi^R =0 &&\text{on }\D\Omega\cap B_R,\\
& \varphi^R=0 &&\text{on }\Omega\cap \D B_R.
\end{aligned}\right.\end{equation}
Note that imposing Dirichlet boundary conditions on $\Omega\cap \D B_R$ is the only way to ensure the decreasing monotonicity of $R\mapsto\lambda_1^R$.

The following result states the existence of the eigenelements for the truncated problem~\eqref{TruncatedProblem} and their convergence when $R\to+\infty$. It implies the existence of a principal eigenfunction in the whole domain. A complete proof can be found in~\cite[Theorem~1.4]{Berestycki2015b} for Dirichlet boundary conditions and in~\cite[Theorem~2.1]{Rossi2020} for Robin boundary conditions (actually, the result is proved for more general \emph{oblique} boundary conditions). 
\begin{lemma}[\cite{Berestycki2015b,Rossi2020}]
\label{lem:preliminary}
Assume that the standing assumptions~(H) hold.
\begin{enumerate}
\item For almost every $R>0$, $\lambda^R_1$ is well defined and admits an eigenfunction $\varphi^R$ which is positive on $\Omega\cap B_R$.
\item $R\mapsto\lambda_1^R$ is strictly decreasing and
\begin{equation}
\lim\limits_{R\to+\infty}\lambda_1^R=\lambda_1.
\end{equation}
\item $\varphi^R$ converges in $C^{2,\alpha}_{loc}$ to some $\varphi$ which is a principal eigenfunction in $\Omega$. 
\end{enumerate}
\end{lemma}

\subsection{Proof of \autoref{th:MPEigen}, first statement}\label{sec:ProofMPEigen}

We divide the proof of the first statement of \autoref{th:MPEigen} in several lemmas which will be useful for the sequel.

The first lemma shows that assumption~\eqref{AssumptionBoundedDrift} entails a variational structure for the operator $\mathcal{L}$ even though it is not self-adjoint.
\begin{lemma}\label{lem:1st}
Assume that~(H) and~\eqref{AssumptionBoundedDrift} hold, let $v$ be a subsolution of $(\mathcal{L},\mathcal{B})$ and let $\lambda_1$ and $\varphi$ be the principal eigenvalue and a principal eigenfunction.
Setting $\sigma:=\frac{v}{\varphi}$, we have
\begin{equation}\label{IdentityVariational}
\nabla\cdot\left(\varphi^2e^\eta A \nabla \sigma\right)\geq \lambda_1e^{\eta}\sigma \varphi^2,\qquad\text{in }\Omega,
\end{equation}
and
\begin{equation}\label{IdentityVariational_Boundary}
\sigma_+\varphi^2\nu\cdot A\nabla\sigma=0,\qquad\text{on } \D\Omega,
\end{equation}
where $\sigma_+=\max(0,\sigma)$ is the positive part of $\sigma$.
\end{lemma}
\noindent
Under Dirichlet boundary conditions, the expression in \eqref{IdentityVariational_Boundary} is not defined since $\varphi=0$ on $\D\Omega$.
In this case, \eqref{IdentityVariational_Boundary} must be understood at the limit when approaching the boundary.
\begin{proof}
Let us first prove~\eqref{IdentityVariational}.
Assumption~\eqref{AssumptionBoundedDrift} directly implies the following identity
\begin{equation}
\nabla\cdot\left(e^{\eta} A\nabla v\right)=\left[\nabla\cdot\left(A\nabla v\right)-B\cdot\nabla v\right]e^\eta.
\end{equation}
Using that $v$ is a subsolution, we deduce
$$
\nabla\cdot\left(e^{\eta} A\nabla v\right)- ce^\eta v\geq 0. 
$$
Similarly, since $\varphi$ is a principal eigenfunction, we have
$$
\nabla\cdot\left(e^{\eta} A\nabla \varphi\right)- ce^\eta \varphi= -\lambda_1e^\eta\varphi. 
$$
Inequality~\eqref{IdentityVariational} is then deduced by a straightforward computation.

\paragraph*{}
Let us now prove~\eqref{IdentityVariational_Boundary}. 
We first consider the case of Robin boundary conditions. In this case, applying Hopf's lemma in the equation for $\varphi$ in~\eqref{EquationVarphi}, we deduce that $\varphi>0$ on $\overline{\Omega}$. Hence, $\sigma$ is bounded. A straightforward computation then gives
\begin{equation}
\nu\cdot A\nabla \sigma=\frac{\nu\cdot A \nabla v}{\varphi}-\sigma\frac{\nu\cdot A\nabla \varphi}{\varphi}\leq-\gamma\sigma+\gamma\sigma=0.
\end{equation}
It proves~\eqref{IdentityVariational_Boundary} in the case of Robin boundary conditions.

Now, consider the case of Dirichlet boundary conditions.
Let $x_0\in\D\Omega$ and set $x_\eps:=x_0-\eps A(x_0)\nu(x_0)$ for all $\eps>0$. If $\sigma(x_\eps)\leq0$ as $\eps$ becomes small, then~\eqref{IdentityVariational_Boundary} trivially holds. Otherwise, we necessarily have that $v(x_0)=0$ and $v(x_{\eps_n})>0$ for a vanishing sequence $(\eps_n)$, therefore $\nu(x_0)\cdot A(x_0)\nabla  v(x_0)\leq 0$. Since $\nu(x_0)\cdot A(x_0)\nabla  \varphi(x_0)>0$ from Hopf's lemma, we deduce that $\lim_{\eps\to0}\sigma(x_\eps)$ exists and equals $\frac{\nu(x_0)\cdot A(x_0)\nabla v(x_0)}{\nu(x_0)\cdot A(x_0)\nabla \varphi(x_0)}$. From this, we deduce
\begin{align*}
&\varphi^2(x_\eps)\sigma_+(x_\eps)\nu(x_0)\cdot A(x_0)\nabla  \sigma(x_\eps)
\\
&=v_+(x_\eps)\big(\nu(x_0)\cdot A(x_0)\nabla v(x_\eps)-\sigma(x_\eps)\nu(x_0)\cdot A(x_0)\nabla \varphi(x_\eps)\big)
\end{align*}
Since $v(x_0)\leq 0$, the right member of the above expression vanishes as $\eps\to0,$ which completes the proof of~\eqref{IdentityVariational_Boundary}.
\end{proof}

Classically, one can multiply~\eqref{IdentityVariational} by $\sigma$ and integrate to derive a variational inequality. However, since the domain is unbounded, we need to introduce a cut-off function.
For $R>0$, we define
\begin{equation}\label{DefCutOffEigenValue}
\chi_R(x):=\chi\left(\frac{\vert x\vert}{R}\right), \quad \forall x\in\R^n,
\end{equation} 
with $\chi$ a smooth nonnegative function such that $${\chi(z)=
\left\{\begin{aligned}
&1 &&\text{if }0\leq z\leq 1,\\
&0 &&\text{if } z\geq 2,
\end{aligned}\right.}\quad \vert \chi'\vert\leq 2.$$
\begin{lemma}\label{lem:2nd}
Assume that~(H) and~\eqref{AssumptionBoundedDrift} holds, let $v$ be a subsolution of $(\mathcal{L},\mathcal{B})$ and $\lambda_1$ be the principal eigenvalue. Then, we have
\begin{equation}
\lambda_1\int_\Omega\chi_R^2e^\eta v_+^2\leq \int_\Omega \vert\nabla\chi_R\vert^2_A e^\eta v_+^2,\qquad \forall R>0,
\end{equation}
where $\vert\nabla\chi_R\vert^2_A=\nabla\chi_R\cdot A \nabla\chi_R$,
and $v_+=\max(v,0)$ is the positive part of $v$.
\end{lemma}
\begin{proof}
Multiplying~\eqref{IdentityVariational} by $\sigma_+\chi_R^2$, integrating over $\Omega$ and using the divergence theorem, we find
\begin{equation}\label{Step1Lemme}
\int_{\D\Omega}\sigma_+\chi_R^2e^\eta\varphi^2\nu\cdot A\nabla\sigma
-\int_\Omega\nabla\left(\chi_R^2\sigma_+\right)\cdot A\nabla\sigma e^\eta\varphi^2
\geq \lambda_1\int_\Omega e^\eta\varphi^2\sigma_+^2\chi_R^2.
\end{equation}
From~\eqref{IdentityVariational_Boundary}, the boundary integral in~\eqref{Step1Lemme} equals zero, 
and so we have
\begin{equation}\label{StepLemma2}
-\int_\Omega\nabla\left(\chi_R^2\sigma_+\right)\cdot A\nabla\sigma e^\eta\varphi^2
\geq \lambda_1\int_\Omega e^\eta\varphi^2\sigma_+^2\chi_R^2.
\end{equation}
Using that 
$$\nabla\left(\chi_R^2\sigma_+\right)\cdot A\nabla \sigma= \left\vert \nabla\left(\chi_R \sigma_+\right)\right\vert^2_A-\vert\nabla\chi_R\vert_A^2\sigma_+^2\geq  -\vert\nabla\chi_R\vert_A^2\sigma_+^2,$$
we find
\begin{equation}
\int_\Omega \vert\nabla\chi_R\vert^2_A\sigma_+^2e^\eta\varphi^2
\geq\lambda_1\int_\Omega e^\eta\varphi^2\sigma_+^2\chi_R^2.
\end{equation}
We conclude the proof using that $\sigma_+\varphi=v_+$.
\end{proof}

The last ingredient that we need for the proof of the first statement of~\autoref{th:MPEigen} is the following technical lemma.
\begin{lemma}\label{lem:technical}
Let $w\not\equiv 0$ be a bounded function. Then,
\begin{equation}
\liminf\limits_{R\to+\infty} R^2\frac{\int_\Omega\vert\nabla\chi_R\vert^2_A w^2}{\int_\Omega \chi_R^2 w^2}<+\infty.
\end{equation}
\end{lemma}
\begin{proof}[Proof of~\autoref{lem:technical}]
By contradiction, assume that there exists $R\mapsto \beta(R)$ positive increasing such that $\beta(+\infty)=+\infty$ and 
\begin{equation}\label{TechnicalLemma1stStep}
R^2\frac{\int_\Omega\vert\nabla\chi_R\vert^2 v^2}{\int_\Omega \chi_R^2 v^2}\geq \beta(R).
\end{equation}
Set $\Omega_R:=\Omega\cap\{\vert x\vert\leq R\}$.
Using that $\chi_R\leq1$, that $\nabla\chi_R$ is supported in $\Omega_{2R}$, and that $\vert\nabla \chi_R\vert^2\leq\frac{1}{R^2}\vert\nabla \chi_1\vert^2$, we deduce from~\eqref{TechnicalLemma1stStep} that
\begin{equation}
\int_{\Omega_R} v^2\leq \frac{K}{\beta(R)}\int_{\Omega_{2R}} v^2,
\end{equation}
for some constant $K>0$ independent of $R$. Iterating this inequality, we find
\begin{equation}
\int_{\Omega_R} v^2\leq \left(\frac{K}{\beta(R)}\right)^j\int_{\Omega_{2^jR}} v^2,
\end{equation}
for all integer $j\geq 1$.
Since $v$ is bounded, we also have that $\int_{\Omega_R} v^2\leq K'R^n$ for some constant $K'>0$ independent of $R$, therefore
\begin{equation}
\int_{\Omega_R} v^2\leq \left(\frac{K}{\beta(R)}\right)^jK'\left(2^jR\right)^n.
\end{equation}
Taking $R>0$ large enough so that $\beta(R)>2^nK$, the right member of the above inequality vanishes as $j\to+\infty$. We find $\int_{\Omega_R} v^2\leq0$, and so $v\equiv0$: contradiction.
\end{proof}

We are now ready to complete the proof of the first statement of \autoref{th:MPEigen}.
\begin{proof}[Proof of \autoref{th:MPEigen}, first statement]
Assume that $v$ is a subsolution of $(\mathcal{L},\mathcal{B})$ and that $v_+\not\equiv0$.  From~\autoref{lem:2nd}, we have that
\begin{equation}
\lambda_1\leq \frac{\int_\Omega\vert\nabla\chi_R\vert^2_A e^\eta v_+^2}{\int_\Omega \chi_R^2 e^\eta v^2_+}.
\end{equation}
Since $w=e^{\frac{\eta}{2}}v_+$ is bounded, \autoref{lem:technical} implies that the right member of the above inequality vanishes along some sequence $R\to+\infty$. We deduce that $\lambda_1\leq 0$, which conclude the proof.
\end{proof}

We point out that the above proof only uses 
$$
\liminf\limits_{R\to+\infty} \frac{\int_\Omega\vert\nabla\chi_R\vert^2_A w^2}{\int_\Omega \chi_R^2 w^2}=0,
$$
which is a weaker statement than the one in \autoref{lem:technical}. Actually, this observation allows us to replace the assumption ``$\lambda_1>0$'' in the first statement of~\autoref{th:MPEigen} by a weaker assumption on the convergence of the sequence of principal eigenvalues on truncated domains. 
\begin{theorem}\label{th:refinement}
Assume~(H) and~\eqref{AssumptionBoundedDrift} hold and recall the definition of $\lambda_1^R$ from~\autoref{lem:preliminary} as the principal eigenfunction for the truncated problem~\eqref{TruncatedProblem}.
If
\begin{equation}\label{WeakPositivityofLambda_1}
\liminf\limits_{R\to+\infty}R^2\lambda_1^R=+\infty,
\end{equation}
then $(\mathcal{L},\mathcal{B})$ satisfies the Maximum Principle in $\Omega$.
\end{theorem}
\begin{proof}
The proof can be deduced from a slight adaptation of above proof of the first statement of \autoref{th:MPEigen}. Namely, replace $\lambda_1,\varphi$ by $\lambda_1^R$ and $\varphi^R$ to deduce
\begin{equation}
\lambda_1^R\int_\Omega\chi_R^2e^\eta v_+^2\leq \int_\Omega \vert\nabla\chi_R\vert^2_A e^\eta v_+^2,\qquad \forall R>0.
\end{equation}
\end{proof}
Note that, since $\lambda^R_1\to \lambda_1$ from \autoref{lem:preliminary}, assumption~\eqref{WeakPositivityofLambda_1} in \autoref{th:refinement} is indeed weaker than the assumption ``$\lambda_1>0$'' in \autoref{th:MPEigen}.

%
%More precisely, denote by $B_R$ the ball of radius $R>0$ and define $\lambda^R_1$ the ``classical'' principal eigenvalue of the truncated eigenvalue problem
%\begin{equation}\left\{\begin{aligned}
%&-\mathcal{L}\varphi^R=\lambda\varphi^R &&\text{in }\Omega\cap B_R,\\
%& \mathcal{B}\varphi^R =0 &&\text{on }\D\Omega\cap B_R,\\
%& \varphi=0 &&\text{on }\Omega\cap \D B_R.
%\end{aligned}\right.\end{equation}
%Note that imposing Dirichlet boundary conditions on $\Omega\cap \D B_R$ is the only way to ensure the decreasing monotonicity of $R\mapsto\lambda_1^R$.
%It is proved in~\cite[Theorem~2.1]{Rossi2020} that $\lambda^R_1$ is well defined for almost every $R>0$ and that
%\begin{equation}
%\lim\limits_{R\to+\infty}\lambda_1^R=\lambda_1.
%\end{equation}
%Then, a slight adaptation in the proofs of \autoref{lem:1st} and \autoref{lem:2nd} (, we have that
%\begin{equation}
%\lambda_1^R\int_\Omega\chi_R^2e^\eta v_+^2\leq \int_\Omega \vert\nabla\chi_R\vert^2_A e^\eta v_+^2,\qquad \forall R>0.
%\end{equation}
%This observation allows to state a refinement of the first statement of \autoref{th:MPEigen}.

\subsection{Proof of \autoref{th:MPEigen}, second statement}

Assume $\lambda_1<0$. From \autoref{lem:preliminary}, there exists $R>0$ such that $\lambda_1^R$ is negative and associated with a principal eigenfunction $\varphi^R$ which solves~\eqref{TruncatedProblem}. 
Note, however, that $\varphi^R$ is not smooth in $\Omega$, otherwise we could directly achieve the proof since $\varphi^R$ is a positive supersolution of $(\mathcal{L},\mathcal{B})$. To bypass this technical difficulty, we use the following technic that was pointed to us by Luca~Rossi. Consider $\Phi(t,x)$ the solution of the evolution problem
\begin{equation}\left\{\begin{aligned}
&\D_t\Phi-\mathcal{L}\Phi=0,&&\forall (t,x)\in(0,+\infty)\times\Omega,\\
&\mathcal{B}\Phi=0, &&\forall (t,x)\in(0,+\infty)\times\D\Omega,\\
&\Phi(t=0,x)=\varphi^R(x), &&\forall x\in\Omega.
\end{aligned}\right.\end{equation}
The function $\varphi^R$ is a generalized supersolution, in the sense that it can be written as the infimum of two supersolutions (namely, $0$ and $\varphi^R$ extended smoothly on $\Omega\cap B_{R+\eps}$). It is then classical that $t\mapsto\Phi(t,\cdot)$ is nonincreasing, i.e., $\D_t\Phi\leq0$, and so for any fixed $t_0>0$, $\Phi(t_0,\cdot)$ is a positive (smooth) supersolution of $(\mathcal{L},\mathcal{B})$ in $\Omega$. It proves that the Maximum Principle does not hold, which achieves the proof.

\subsection{Proof of \autoref{th:CriticalMP}}\label{sec:Proof_CriticalMP}

\begin{proof}[Proof of \autoref{th:CriticalMP}]
It can be seen from the proof of the second statement of~\autoref{th:MPEigen} that, if $\lambda_1<0$, then $(\mathcal{L},\mathcal{B})$ does not satisfy the {Critical Maximum Principle} in $\Omega$. It proves the second statement of \autoref{th:CriticalMP}.
\paragraph*{}
We now turn to the proof of the first statement of \autoref{th:CriticalMP}.
Let us assume that~\eqref{AssumptionBoundedDrift} and~\eqref{GrowthConditionEigen} hold, that $\lambda_1\geq0$, and let us prove that the Critical Maximum Principle holds. Let $v$ be subsolution of $(\mathcal{L},\mathcal{B})$ with finite supremum, and set $\sigma:=\frac{v}{\varphi}$, where $\varphi$ is a principal eigenfunction (given by \autoref{lem:preliminary}). 
Our goal is to show that $\sigma_+=\max(\sigma,0)=\frac{v_+}{\varphi}$ is constant.

Let us consider the cut-off function $\chi_R$ defined in~\eqref{DefCutOffEigenValue} for $R>0$.
From inequality~\eqref{StepLemma2} in the proof of~\autoref{lem:2nd}, we have
\begin{equation}
\int_\Omega\nabla\left(\chi_R^2\sigma_+\right)\cdot A\nabla\sigma e^\eta\varphi^2
\leq 0.
\end{equation}
By expanding the term $\nabla\left(\chi_R^2\sigma_+\right)$, we find
\begin{equation}
\int_\Omega\vert\nabla\sigma_+\vert^2_A\chi_R^2e^\eta\varphi^2
\leq -2\int_\Omega\nabla\chi_R\cdot A\nabla\sigma_+\sigma_+\chi_Re^\eta\varphi^2.
\end{equation}
Using Cauchy Schwartz inequality and that $\nabla\chi_R$ is supported in $\{R\leq \vert x\vert \leq 2R\}$, we obtain
\begin{equation}\label{StepTheoremDegenere}
\int_\Omega\vert\nabla\sigma_+\vert^2_A\chi_R^2e^\eta\varphi^2
\leq 2\sqrt{\left(\int_{\Omega\cap\{R\leq \vert x\vert\leq  2R\}}\vert\nabla\sigma_+\vert^2_A\chi_R^2e^\eta\varphi^2\right)\left(\int_\Omega\vert\nabla\chi_R\vert^2_A\sigma_+^2e^\eta\varphi^2\right)}.
\end{equation}
Since $\sigma_+\varphi=v_+$ and $e^\eta$ are bounded, and since $\vert\nabla \chi_R\vert^2=\frac{1}{R^2}\vert \nabla\chi_1\vert^2$, we deduce
\begin{equation}
\int_\Omega\vert\nabla\chi_R\vert^2_A\sigma_+^2e^\eta\varphi^2\leq \frac{K}{R^2}\big\vert\Omega\cap\{R\leq \vert x\vert\leq  2R\}\big\vert,
\end{equation}
for some constant $K>0$ independent of $R$. 
Assumption \eqref{GrowthConditionEigen} implies that the right member of the above inequality is bounded uniformly in $R$. Injecting this estimate in~\eqref{StepTheoremDegenere}, we obtain
\begin{equation}\label{StepTheoremDegenere2}
\int_\Omega\vert\nabla\sigma_+\vert^2_A\chi_R^2e^\eta\varphi^2
\leq K\sqrt{\int_{\Omega\cap\{R\leq \vert x\vert\leq  2R\}}\vert\nabla\sigma_+\vert^2_A\chi_R^2e^\eta\varphi^2},
\end{equation}
for some constant $K>0$ independent of $R$.
We deduce that $\int_\Omega \vert\nabla\sigma_+\vert^2_A\chi_R^2e^\eta\varphi^2$ is bounded uniformly in $R$. We can therefore pass to the limit in~\eqref{StepTheoremDegenere2} as $R\to+\infty$, which gives
${
\int_\Omega\vert\nabla\sigma_+\vert^2_A\varphi_1^2e^\eta\leq0.
}$
Hence $\nabla\sigma_+=0$, which ends the proof.
\end{proof}

The core of the proof consists in showing that $\nabla\cdot(\varphi_1^2A\nabla\sigma)\geq0$ (from \autoref{lem:1st}) implies $\nabla \sigma_+=0$.
The literature refers to this property as a \emph{Liouville property}. Originally introduced by Berestycki, Caffarelli, Nirenberg in~\cite{Berestycki1997b}, it has been extensively used and discussed~\cite{Barlow2000,Gazzola,Moschini2005,Villegas2020,Karp1981}.
Actually, we believe that the technics of~\cite[Theorem~5.1]{Moschini2005} allow to relax assumption~\eqref{GrowthConditionEigen} to
$$
\left\vert\Omega\cap\{\vert x\vert\leq R\}\right\vert = O(R^2\log(R)).
$$
However, a recent work of Villegas~\cite{Villegas2020}  suggests that the conclusions of \autoref{th:CriticalMP} do not hold if we only assume 
$$\left\vert\Omega\cap\{\vert x\vert\leq R\}\right\vert =O(R^2\log(R)^2).$$
Hence, assumption~\eqref{GrowthConditionEigen} seems to be essentially optimal in our context. See~\cite{Villegas2020,Moschini2005} for more details.
% For more details, we refer to our previous paper~\cite[Section~3.2]{Nordmann2019a}.

\subsection{Proof of \autoref{cor:VarphiBounded}}
Assume that $\lambda_1=0$ and let $\varphi$ be a principal eigenfunction. If $\varphi$ is bounded, then it is a bounded positive supersolution and so the Maximum Principle does not hold: it proves the first statement.

Conversely, let $v$ be subsolution with a finite supremum. Under assumption~\eqref{GrowthConditionEigen} \autoref{th:CriticalMP} implies that $v$ is either nonpositive or is a constant multiple of $\varphi$. However, since $v$ is bounded and $\varphi$ is not, $v$ cannot be a non-zero multiple of $\varphi$. Therefore, $v$ is nonpositive and the Maximum Principle holds.

\subsection{Proof of \autoref{th:Simplicity}}
Let $\lambda_1$ be the principal eigenvalue and $\varphi,\psi$ be two principal eigenfunctions. Up to replacing $\mathcal{L}$ with $\mathcal{L}+\lambda_1$, we can assume without loss of generality that $\lambda_1=0$.
Since $\psi$ is bounded, it is a positive subsolution of $(\mathcal{L},\mathcal{B})$ with finite supremum. Then, \autoref{th:CriticalMP} implies that $\psi$ is a constant multiple of $\varphi$, therefore $\lambda_1$ is simple.

\subsection{Proof of~\autoref{th:EigenMPNOTSelfAdjoint}}
\begin{proof}[Proof of~\autoref{th:EigenMPNOTSelfAdjoint}]
The proof of \autoref{th:EigenMPNOTSelfAdjoint} for Dirichlet boundary conditions is already contained in~\cite[Theorem~1.6]{Berestycki2015b}. We therefore focus on the case of Robin boundary conditions.
 
We assume $\tilde \lambda_1>0$. Let $v$ be a subsolution of $(\mathcal{L},\mathcal{B})$ with finite supremum, and $\varphi$ be a supersolution of $(\mathcal{L}+\lambda,\mathcal{B})$, $\lambda\in(0,\tilde{\lambda}_1)$ with positive infimum. Let us show $v\leq0$. Up to renormalization, we can assume without loss of generality that $\sup_\Omega v\leq 1$ and $\inf_\Omega\varphi\geq 3$.

The function $v$ and $\varphi$ satisfy differential inequalities which do not yield uniform $C^{2,\alpha}$ estimates. In the following technical lemma, we construct two auxiliary functions $\overline{u}\geq\underline{u}\geq v$ which satisfy the same differential inequalities as $v$ and $\varphi$ respectively and are also solutions to some semilinear elliptic equations that entail uniform estimates.
\begin{lemma}\label{th:LemmaEigenNonHomogene}
There exist two functions $\underline{u},\overline{u}\in C^2(\overline \Omega)$ such that $$v_+\leq \underline u\leq 2\leq \overline u\leq \varphi$$ where $v_+=\max(v,0)$, and
\begin{equation}\label{Equations_2u}
\left\{\begin{aligned}
&\mathcal{L}\overline{u}=\vert c(x)\vert \overline\theta(\overline{u})+\lambda \overline{u}&&\Omega,\\
&\mathcal{B}\overline{u}=\vert \gamma(x)\vert \overline \theta(\overline{u}) &&\D\Omega,
\end{aligned}\right.
\quad
\left\{\begin{aligned}
&\mathcal{L}\underline{u}=\vert c(x)\vert \underline\theta(\underline{u})&&\Omega,\\
&\mathcal{B}\underline{u}=\vert \gamma(x)\vert \underline \theta(\underline{u}) &&\D\Omega,
\end{aligned}\right.
\end{equation}
where
\begin{equation}\label{property_theta_overline}
\overline \theta(\cdot)\text{ is smooth, nonnegative, nonincreasing},\ \overline\theta =
\left\{\begin{aligned}
&1 &&\text{ on }(-\infty,2],\\
&0&&\text{ on }[3,+\infty),
\end{aligned}\right.
\end{equation}
\begin{equation}
\underline \theta(\cdot)\text{ is smooth, nonpositive, nonincreasing},\ \underline\theta =
\left\{\begin{aligned}
&0&&\text{ on }(-\infty,1],\\
&-1&&\text{ on }[2,+\infty).
\end{aligned}\right.
\end{equation}
\end{lemma}
\noindent
The proof of this lemma is postponed at the end of the section.

Let us go on with the proof of \autoref{th:EigenMPNOTSelfAdjoint}. We set
\begin{equation}
t_0:=\inf \left\{t\geq 0: \underline{u}\leq t\overline{u}\right\}.
\end{equation}
By contradiction, assume $t_0>0$. Denoting $w:= \underline{u}-t_0\overline u$, we have that $w\leq 0$ and that there exists a sequence $x_n\in\Omega$ such that $w(x_n)\to \sup w= 0$.
Intuitively, our goal is to obtain a contradiction from taking the limit as $n\to+\infty$ and applying the Maximum Principle.
Up to extraction of a subsequence, we can assume without loss of generality that either $d(x_n,\D\Omega)\to0$ or $\liminf d(x_n,\D\Omega)>0$.
Set $\overline u_n(\cdot):=\overline u(\cdot+x_n)$, $\underline u_n(\cdot):=\underline u(\cdot+x_n)$, $w_n(\cdot):=w(\cdot+x_n)$, defined in the closure of $\Omega_n:=\Omega-x_n$. Let $B_1$ denote the unit ball centered at the origin and set $V_n:=\overline \Omega_n\cap B_1$.
%and $\Sigma_n:=\D\Omega\backslash\D B_1$. %Up to extraction, we know that $\tilde \Omega_n$ converges to some $\tilde\Omega_\infty$. We denote $\Sigma_\infty:=\D\tilde\Omega_\infty\backslash$
Let also $R>0$ be the radius in the definition of the uniform $C^2$ regularity of the domain~\eqref{UniformlyC2}, and let $y_n\in\Omega_n$ be a sequence such that $0\in B_n:=\{\vert x-y_n\vert\leq R\}\subset\Omega_n$.
%Note that, up to extraction of a subsequence, we can choose $y_n$ such that $\vert y_n\vert>\frac{\liminf d(0,\D\Omega_n)}{2}$ (intuitively, this means that if $x_n$ remains far from the boundary of $\Omega$, then so does $B_n$).
%
%
%, if $\liminf d(0,\D\Omega_n)>0$, then $\limsup\vert y_n\vert<R$ (intuitively, this means that if $x_n$ remains far from the boundary of $\Omega$, then so does $B_n$).
Without loss of generality, we can also assume that $R$ is so small that $\overline{B}_n\subset V_n$.

Let us derive uniform $C_{loc}^{2,\alpha}$ estimates on $\underline{u}_n$ and $\overline{u}_n$.
For $n$ large enough, we have that $w_n(0)\geq -1$, therefore $\inf_{V_n} u_n$ is bounded from above uniformly in $n$.
From the classical Harnack inequality, $\overline{u}$ satisfies
\begin{equation}\label{Harnack_Classical}
\sup_{V_n}\overline u_n\leq C \inf_{V_n}\overline u_n.
\end{equation}
with a constant $C>0$ independent of $n$. This implies that $\sup_{V_n}\overline u_n$ is bounded uniformly in $n$. Then, classical Schauder estimates imply that $\overline{u}_n$, $\underline{u}_n$, and $w_n$ are bounded in $C^{2,\alpha}(V_n)$, uniformly in $n$.

The uniform $C^{2,\alpha}$ imply by compactness that $w_n$ converges (up to a subsequence) to some $w_\infty$ in $C^2(\overline{B}_\infty)$, where $B_\infty$ is a ball of radius $R$ and $0\in\overline{B}_\infty$. Similarily, the coefficients $\mathcal{L}_n:=\mathcal{L}(\cdot+x_n)$, $\mathcal{B}_n=\mathcal{B}(\cdot+x_n)$ converge in $C^{0,\alpha}$ (up to a subsequence) to some $\mathcal{L}_\infty$, $\mathcal{B}_\infty$.
We further have
\begin{equation}\label{Ineq_w_infty}
w_\infty\leq0=w_\infty(0),
\end{equation}
 and
\begin{equation}\label{EigenEqAsymptotic}
\left\{\begin{aligned}
&\mathcal{L}_\infty w_\infty \leq -t_0\lambda \overline{u}_\infty &&\text{in }B_\infty,\\
&\mathcal{B}_\infty w_\infty\leq 0 &&\text{in }\D\Omega\cap \overline{B_\infty}.
\end{aligned}\right.
\end{equation}

Let us show that $w_\infty\equiv0$. On the one hand, if $0\in B_\infty$ then the classical Strong Maximum Principle implies that $w_\infty\equiv0$. On the other hand, if $0\in\D\Omega\cap B_\infty$ then the boundary condition in~\eqref{EigenEqAsymptotic} yields 
\begin{equation}
\nu(0)\cdot A_\infty(0)\nabla w_\infty(0)\leq -\gamma_\infty(0)w_\infty(0)= 0.
\end{equation}
From Hopf's lemma, we deduce $w_\infty\equiv0$. In both cases, we have derived that $w_\infty\equiv0$.

Using~\eqref{EigenEqAsymptotic}, that $t_0>0$, and that $\inf\overline{u}_\infty>0$, we deduce $\lambda=0$: contradiction. Thus $t_0=0$, $\underline{u}=v_+=0$, and so $v\leq0$.
\end{proof}

\begin{proof}[Proof of \autoref{th:LemmaEigenNonHomogene}]
The proof is inspired by the proof of~\cite[Proposition~5.2]{Berestycki2015b}.
We use the notations
\begin{equation}
f(x,s):= \vert c(x)\vert\overline{\theta}(s)+\lambda s\quad ; \quad g(x,s):=\vert \gamma(x)\vert \overline{\theta}(s).
\end{equation}
Our goal is to show the existence of a solution $\underline{u}$ of the equation
\begin{equation}\label{equation_underline_u}
\left\{\begin{aligned}
&\mathcal{L}\underline{u}= f(x,\underline{u})&&\Omega,\\
&\mathcal{B}\underline{u}=g(x,\underline{u}) &&\D\Omega.
\end{aligned}\right.
\end{equation}
which also satisfies $2\leq \overline{u}\leq \varphi$. Our arguments can be adapted without difficulty to prove the existence of $\underline{u}$ satisfying the required conditions.

From the assumptions~\eqref{property_theta_overline} on $\overline{\theta}$ and that $\inf_\Omega\varphi\geq3$, we have that
\begin{equation}\label{SupEquation_Phi}
\left\{\begin{aligned}
&\mathcal{L}\varphi\geq f(x,\varphi)&&\Omega,\\
&\mathcal{B}\varphi\geq g(x,\varphi) &&\D\Omega.
\end{aligned}\right.
\end{equation}
Setting $\sigma\equiv 2$, we also have that
\begin{equation}\label{SubEquation_Sigma}
\left\{\begin{aligned}
&\mathcal{L}\sigma\leq f(x,\sigma)&&\Omega,\\
&\mathcal{B} \sigma\leq g(x,\sigma) &&\D\Omega.
\end{aligned}\right.
\end{equation}
We are going to use $\varphi$ and $\sigma$ as a super and a subsolution of~\eqref{equation_underline_u}, and construct $\overline u$ with Perron's iterative method on a truncated domain. Let us consider a sequence $0< R_j\to+\infty$ of real positive numbers and an increasing sequence of bounded Lipschitz subdomains $\Omega^{R_j}\subset\Omega$ such that $\bigcup\limits_{j>0}\Omega^{R_j}=\Omega$. We denote $\Sigma^{R_j}=\D\Omega^{R_j}\cap\D\Omega$, and we choose $\Omega_R$ such that the intersection of $\Sigma_R$ and $\D\Omega_R\backslash\Sigma_R$ is a $C^2$ $(n-2)$-dimensional manifold.

We set $u_0=\varphi$ and define by induction $u_{n+1}$ as the unique solution of
\begin{equation}\label{MixedBoundary_recu}
\left\{\begin{aligned}
&\mathcal{L} u_{n+1}-C u_{n+1}= f(x,u_n)- C u_n&&\Omega_R,\\
&\mathcal{B} u_{n+1}-\Gamma u_{n+1}= g(x,u_n)-\Gamma u_{n} &&\Sigma_R,\\
& u_{n+1}=\varphi &&\D\Omega_R\backslash\Sigma_R
\end{aligned}\right.
\end{equation}
with $C:=\inf_\Omega c$ and $\Gamma:=\inf_{\D\Omega}\gamma$. 
From the results of Liberman~\cite{Lieberman1986}, we know that all classical results (Schauder estimates, Maximum Principle, solvability, etc.) hold from the mixed boundary value problem~\eqref{MixedBoundary_recu}. First, those results imply that the sequence $u_n$ is well defined. Then, from the Maximum Principle, we can show by induction that 
\begin{equation}\label{inegalite_recurrence}
\sigma\leq u_{n+1}\leq u_n\leq \varphi.
\end{equation}

From the a priori Schauder estimates proved in~\cite{Lieberman1986} for the mixed boundary value problem~\eqref{MixedBoundary_recu}, we know that $u_n$ is bounded in $C^{2,\alpha}$ uniformly in $n$, and therefore converges in $C^2(\overline \Omega_R )$ to some function $u^R$ which is a solution of
\begin{equation}\label{EquationUR}
\left\{\begin{aligned}
&\mathcal{L} u^R= f(x,u^R)&&\Omega_R,\\
&\mathcal{B} u^R= g(x,u^R)&&\Sigma_R,\\
& u^R=\varphi &&\D\Omega_R\backslash\Sigma_R
\end{aligned}\right.
\end{equation}
We also know from~\eqref{inegalite_recurrence} that
\begin{equation}
\sigma\leq u^R\leq \varphi.
\end{equation}

From Theorem 3.3 in~\cite{Lieberman1987} and Theorem 4.3 in~\cite{Lieberman2001}, we can show that the Harnack estimate holds for the mixed boundary problem~\eqref{EquationUR}, namely, we have that 
\begin{equation}
\sup_{\overline\Omega_{R_0}} u^R\leq C\inf_{\overline\Omega_{R_0}} u^R\leq C\inf_{\overline \Omega_{R_0}} \varphi,
\end{equation}
where $C$ is a constant independent of $R$.

Now, from classical Schauder estimates, $u^R$ is uniformly $C^{2,\alpha}$ in $\overline{\Omega}_{R_0}$. Thus, $u^R$ converges (up to extraction) to some $\overline u$ in $C_{loc}^2(\overline{\Omega})$ when $R\to+\infty$, which satisfies the required conditions.

Let us now derive an estimate on $u^R$ which is uniform in $R$. We fix $R_0>0$ and take $R>R_0$. 
\end{proof}

\paragraph*{Aknowledgments.} The author is deeply thankful to Professor Luca Rossi and to Professor Yehuda Pinchover for very useful comments and discussions.

\bibliographystyle{plain}
\bibliography{/Users/samuelnordmann/Dropbox/Etudes/Bibliographie/library}
\end{document}